\documentclass{amsart}
\usepackage[T1, T2A]{fontenc}
\usepackage[utf8]{inputenc} 
\usepackage{graphicx}

\usepackage{amssymb,latexsym}
\usepackage{subfigure, url}
\usepackage{soul}
\usepackage[section]{placeins}

\usepackage{tikz}
\usepackage{color}
\definecolor{darkgreen}{RGB}{0,100,0}
\def\ov{\overline}
\let\del=\partial
\newcommand{\eps}{\varepsilon}

\def\D{\mathbb D} 
\def\Cc{{\mathbb C}}

\newcommand{\R}{\mathbb R}
\newcommand{\N}{\mathbb N}

\newcommand\inv{^{-1}}
\newcommand\otau{{\overline\tau}}
\let\ol=\overline
\DeclareMathOperator\sgn{sgn}
\DeclareMathOperator\Fix{Fix}
\DeclareMathOperator\Res{Res}
\DeclareMathOperator\dom{dom}
\def\Hh{{\mathbb H}}
\theoremstyle{plain}
\newtheorem{lemma}{Lemma}[section]
\newtheorem{proposition}[lemma]{Proposition}
\newtheorem{theorem}[lemma]{Theorem}
\newtheorem{corollary}[lemma]{Corollary}
\newtheorem{remark}[lemma]{Remark}
\newtheorem{remarks}[lemma]{Remarks}
\newtheorem{notation}[lemma]{Notation}
\newtheorem{definition}[lemma]{Definition}

\catcode`@=11
\renewcommand\paragraph{\@startsection{paragraph}{4}\z@{5pt}{-\fontdimen 2\font }\bfseries}
\catcode`@=12

\title[Real polynomial vector fields]{Generic complex polynomial vector fields with real coefficients}
\author[J. Godin \& C. Rousseau]{Jonathan Godin and Christiane Rousseau}
\address{Universit\'e de Moncton, Campus de Shippagan, 218 boul. J.-D.-Gauthier, Shippagan (NB), E8S 1P6, Canada --- D\'epartement de math\'ematiques et de statistique, Universit\'e de Montr\'eal, C.P. 6128, Succursale Centre-ville, Montr\'eal (QC), H3C 3J7, Canada.}
\email{ jonathan.godin@umoncton.ca --- christiane.rousseau@umontreal.ca}

\usepackage[hidelinks]{hyperref}
\hypersetup{
	pdfauthor={Jonathan Godin, Christiane Rousseau},
  pdftitle={Generic complex polynomial vector fields with real coefficients},
  pdfsubject={Classification of polynomial vector fields with real coeficients},
  pdfkeywords={Polynomial vector fields, generic vector fields,
		counting generic strata, combinatorial invariant, analytic invariants,
		realization}
}

\begin{document} 
\date{\today}
\maketitle

\begin{abstract} The paper studies the complex 1-dimensional polynomial vector fields with real coefficients under topological orbital equivalence preserving the separatrices of the pole at infinity. The number of  generic strata is determined, and a complete parametrization of these strata is given in terms of a modulus formed by a combinatorial and an analytic part. The bifurcation diagram is described for the degree 4. A realization theorem is proved for any generic modulus. \end{abstract}


\section{Introduction} 

The study of the dynamics of 1-dimensional holomorphic maps involves a mixture of local and global techniques. Among the local techniques is the study of the multiple fixed points (also called \emph{parabolic points}) of germs of holomorphic diffeomorphisms and their unfoldings. For instance, the rightmost real point of the Mandelbrot set corresponds to a quadratic map with a parabolic point. 
For a parabolic  fixed point of multiplicity $k+1$, i.e. codimension $k$, it is natural to embed the corresponding germ of diffeomorphism into a generic family depending on $k$ parameters. Indeed, the parabolic points organize the dynamics in their neighborhoods inside the generic families. 

Providing a complete description of the dynamics of unfoldings of parabolic points proved to be a very difficult problem. A simple formal normal form was known consisting of the time-one map of a rational vector field. But the change of coordinate to normal form was divergent. For parameter values for which it is possible to bring the system to the normal form in the neighbordood of the unfolded simple fixed points, the obstruction to convergence was explained by the fact that normalizations did not match globally. But nobody knew how to treat the case where the unfolded simple fixed points were not linearizable. The idea came from Douady, namely to normalize on domains with two sectors adherent to two different fixed points. It  was used by Lavaurs and Oudkerk to treat some sectoral regions where the unfolded simple fixed points were not linearizable (\cite{D94}, \cite{L89} and \cite{O99}). The visionary paper of Douady-Estrada-Sentenac \cite{DES05} made the breakthrough and opened the way to treat all parameter values by the same method (for a non exhaustive list see \cite{MRR04}, \cite{Ri08},  \cite{Ro15}). 
The shape of the domains of normalization is  controlled by the dynamics of complex 1-dimensional polynomial vector fields. The paper \cite{DES05} contains a complete study of the \lq\lq generic\rq\rq\ complex polynomial vector fields, meant here to be the structurally stable ones. It was followed by other works studying the non-generic cases (see for instance \cite{BD10}, \cite{DT16} and \cite{D20}). 

Holomorphic dynamics develops in parallel to antiholomorphic dynamics, and in the latter case parabolic points do also play an important role. The second iterate of an antiholomorphic parabolic germ is a holomorphic parabolic germ of the same multiplicity (see for instance \cite{HS14}, \cite{ IM16} and \cite{GR22}). However, the natural parameters unfolding an antiholomorphic parabolic germ of multiplicity $k+1$ are real. An analytic  classification of generic unfoldings of antiholomorphic parabolic germs of multiplicity $k+1$ depending real-analytically on $k$ parameters can be obtained by extending these germs antiholomorphically in the parameters (the idea of extending antiholomorphically was already mentioned by Milnor in \cite{M92}). If $f_\eps$ is such a germ, then the map $g_\eps= f_{\ov{\eps}}\circ f_\eps$ is a generic unfolding of a holomorphic parabolic germ and its modulus is a modulus for $f_\eps$ (\cite{GR23} and \cite{Ro23b}). But the construction of the modulus appearing there does not respect the real character of the parameters. A description respecting the real character of the parameters is still needed and will  involve an underlying vector field with real coefficients. This is the motivation for this study. 

The results can also be helpful for any problem where parabolic points naturally depend on real parameters. This is the case of the Poincar\'e return map in the Hopf bifurcation. For the codimension 1 case, it  was shown in \cite{A12} that the orbital classification of families of vector fields unfolding a Hopf bifurcation follows from the conjugacy of the Poincar\'e return maps. 

The paper describes the complete dynamics of generic polynomial vector fields $P(z) \frac{\partial }{\partial z}$, where $P$ is a polynomial of degree $k+1$ with real coefficients. 
In practice, this means providing the phase portrait of the ODE $\dot z = \frac{dz}{dt} = P(z)$  for real time $t$. Modulo an affine change of coordinate we can always suppose that the polynomial $P$ has the form
\begin{equation}P(z) = \begin{cases}
z^{k+1} + \eps_{k-1}z^{k-1} + \dots + \eps_1z+\eps_0, &k\ \text{odd},\\
\pm z^{k+1} + \eps_{k-1}z^{k-1} + \dots + \eps_1z+\eps_0, &k\ \text {even},\end{cases}\label{eq:P} \end{equation} 
and we also  note it $P_\eps$ with $\eps= (\eps_{k-1}, \dots, \eps_1, \eps_0)\in \R^k$, to emphasize its coefficients. 
We only consider the $+$ case. The other case can be obtained from it by changing $t\mapsto -t$. 

\begin{remark} When complex coefficients are allowed and $k$ is 
 odd, the two cases $\pm$ are equivalent under a change $z\mapsto
\exp\left(\frac{\pi i \ell}{k}\right)z$, $\ell$ odd.
\end{remark}

\begin{notation} 
We call $\mathcal{P}_{\Cc,k+1}$ (resp. $\mathcal{P}_{\R, k+1}$) the family of polynomial vector fields $X_\eps=P_\eps(z) \frac{\partial}{\partial z}$ with complex (resp. real) coefficients, where 
$$P_\eps(z) = z^{k+1} + \eps_{k-1} z^{k-1} + \dots + \eps_1z+\eps_0$$ 
and $\eps\in \Cc^k$ (resp.  $\eps\in \R^k$).
\end{notation}

\begin{remark}\label{strata_parameter} In the whole paper, we identify $\mathcal{P}_{\Cc,k+1}$ (resp. $\mathcal{P}_{\R, k+1}$) with the topological space $\Cc^k$ (resp. $\R^k$).\end{remark}

\emergencystretch=1em
In this paper, we determine the equivalence classes of generic complex 1-dimensional polynomial vector fields with real coefficients under topological orbital equivalence preserving the separatrices of the pole at infinity: these equivalence classes are called generic \emph{strata}. In Section~\ref{sec:preliminaries} we recall the known results for $\mathcal{P}_{\Cc, k+1}$, and in Section~\ref{sec:P_R}, we derive the corresponding results for $\mathcal{P}_{\R, k+1}$, namely each stratum is characterized by a combinatorial invariant and parameterized by an  analytic invariant. In Section~\ref{sec:generic_strata}, we determine  the exact number of  generic strata in $\mathcal{P}_{\R,k+1}$. In Section~\ref{sec:structure}, we explore the structure of the family $\mathcal{P}_{\R,k+1}$. In particular we show that it contains a complete unfolding of any non real parabolic point. In Section~\ref{sec:k=3}, we determine the bifurcation diagram of polynomial vector fields in $\mathcal{P}_{\R,k+1}$ of degree 4. Finally, in Section~\ref{sec:real}, a realization theorem is proved for any generic combinatorial and analytic invariant. 
\par 
\emergencystretch=0pt

\section{Preliminaries on polynomial vector fields in $\mathcal{P}_{\Cc, k+1}$}\label{sec:preliminaries} 

In this section we briefly recall the results of \cite{DES05} on polynomial vector fields~$\mathcal{P}_{\Cc, k+1}$. 

\begin{remark}\label{rem:scaling} The change \begin{equation}(z, \eps_{k-1}, \dots, \eps_1,\eps_0,t)\mapsto (zr, \eps_{k-1} r^{-(k-2)}, \dots, \eps_1,\eps_0r,tr^k)\label{scaling}\end{equation} sends a vector field to one with the same phase portrait modulo a zoom.  This allows using scalings when discussing particular situations. 
\end{remark}

\subsection{The dynamics of polynomial vector fields $\mathcal{P}_{\Cc, k+1}$}

The dynamics is governed by the singular point at $\infty$, which is a pole of order $k-1$ (see Figure~\ref{infinity}(a)). The pole has the behavior of a saddle with $2k$ sepatrices.  This means that near the boundary of any sufficiently large disk containing all the singular points, the trajectories are organized like a flower (see Figure~\ref{infinity}(b)).
Note that $\infty$ is reached in finite time. Moreover,  the complexified time is ramified at the pole and covering $k$ times a neighborhood of $\infty$. \begin{figure}
\begin{center}\subfigure[]{\includegraphics[width=5cm]{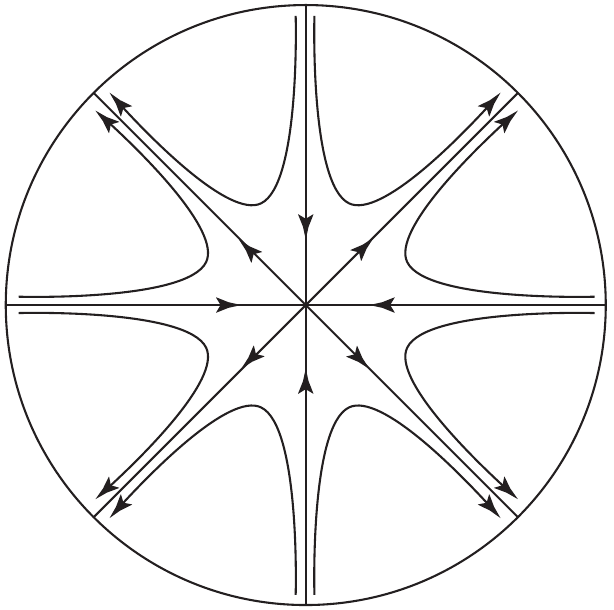}}\qquad\subfigure[]{\includegraphics[width=5cm]{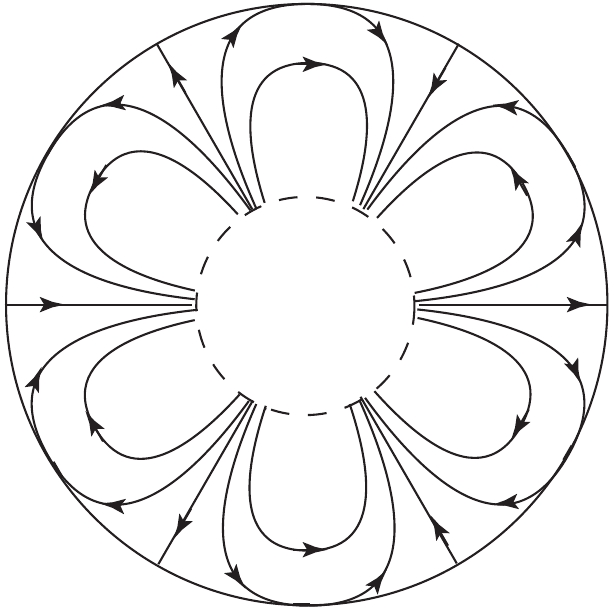}}\caption{(a)
The phase portrait at infinity. (b) The phase portrait near the
boundary of a large disk containing all singular points.}
 \label{infinity}\end{center}\end{figure}

\begin{proposition} \cite{DES05} Let  $X_\eps  \in\mathcal{P}_{\Cc, k+1}$.
\begin{enumerate} 
  \item Any separatrix of $\infty$ either lands at a singular
    point or merges with another separatrix to form a homoclinic loop.  
  \item Any simple finite singular point is an antisaddle, either a radial node, or a
    strong focus, or a center. A singular point is a center if and only if its
    eigenvalue is pure imaginary.  
  \item There are no limit cycles.  
  \item Any trajectory is either: 
  \begin{itemize} 
    \item a fixed point, 
    \item a separatrix of $\infty$ landing at a singular point, 
    \item a homoclinic loop through infinity,
    \item a periodic orbit of a period annulus, 
    \item a trajectory whose $\alpha$- and $\omega$- limits are singular
      points. In the case where the $\alpha$- and $\omega$- limits are the same
      singular point, then this point is a parabolic (multiple) point.
  \end{itemize} 
  \item Any center is surrounded by a
period annulus, whose boundary is a union of homoclinic loops through infinity.
  \item There is at least one separatrix landing at each radial node or focus.
  \item If $z_1, \dots, z_\ell$ are the (possibly multiple) singular points of a vector field
    $X_\eps  = P_\eps \frac\del{\del z}$, then $\sum_{j=1}^{\ell}
  \Res\bigl( \frac1{P_\eps}, z_j\bigr) =0$.   \item Any homoclinic loop surrounds a group of singular points $z_{j_1}, \dots, z_{j_s}$
    such that $\sum_{\ell=1}^{s} \Res\bigl(\frac1{P_\eps},z_{j_\ell}\bigr) \in i\R^*$. \end{enumerate} \end{proposition}

\subsection{DES-equivalence and structurally stable vector fields}

\begin{definition} \begin{enumerate} \item Two vector fields $X_\eps$ and $X_\eps'$ are said to be \emph{DES-equivalent} if there exists a homeomorphism of $\Cc$ preserving the separatrices at infinity, which is a topological orbital equivalence between the vector fields, i.e. it sends the trajectories to trajectories and preserves the orientation, but not necessarily the parametrization. 
\item Each equivalence class (or the corresponding set of parameter values, see Remark~\ref{strata_parameter}) is called a  \emph{stratum}.
\item The strata of full dimension are open sets of the parameter space and called \emph{generic strata}, and the vector fields of a generic stratum are called  \emph{generic} (in the sense of Douady-Estrada-Sentenac).  Each vector field $X_\eps$ belonging to a generic stratum is \emph{structurally stable}, i.e. DES-equivalent to all vector fields in a neighborhood of $X_\eps$. \end{enumerate} \end{definition}

\begin{remark} Each stratum is a submanifold of the parameter space. Its boundary is a finite union of strata of lower dimension.\end{remark} 

The bifurcation diagram for this equivalence relation divides the parameter space into open regions (generic strata) of structurally stable vector fields separated by bifurcation \lq\lq surfaces\rq\rq, which are unions of strata of lower dimension.  
The only bifurcations of real codimension 1 are the homoclinic loops through infinity. The higher real codimension bifurcations are the parabolic points (the real codimension is $2s$ for a multiple point of multiplicity $s+1$) and intersections of lower codimension bifurcations.

\begin{proposition} A vector field $X_\eps\in\mathcal{P}_{\Cc, k+1}$ is generic
 if and only if it has only simple singular points and no homoclinic loop.
\end{proposition}

\subsection{The combinatorial part of the invariant of a structurally stable vector field}
\begin{proposition}\label{prop:tree-graph}\cite{DES05} 
Let $X_\eps \in\mathcal{P}_{\Cc, k+1}$ be a generic vector field. Two trajectories having the same $\alpha$- and $\omega$-limits are called \emph{equivalent}: this defines an equivalence relation on trajectories whose $\alpha$- and $\omega$-limits are singular points. 
Let us define the following (non oriented) graph:
\begin{itemize}
\item the vertices are the singular points;
\item there is an edge between two vertices if and only if there is a trajectory linking the corresponding singular points. Then the edge is given by the equivalence classes of trajectories having these two vertices as $\alpha$- and $\omega$-limits.
\end{itemize}
Then this graph is a planar tree graph. On this graph, the only isomorphisms we allow are the ones provided by orientation preserving homeomorphisms of the plane. The graph is attached to the $2k$ separatrices in the following way: if we turn around the graph along a closed curve, then we meet $2k$ singular points. The separatrices are attached in a one-to-one manner to these $2k$ points in a non-crossing way.  
(See Figure~\ref{fig:tree graph}.)
This attachement is uniquely determined by the attachment of a first separatrix to a singular point. The tree graph of a vector field and its attachment to the separatrices is called the \emph{combinatorial part} of the invariant of $X_\eps$. All generic vector fields of a given stratum have the same combinatorial part.  \end{proposition}

\begin{figure}
  \centering
  \includegraphics{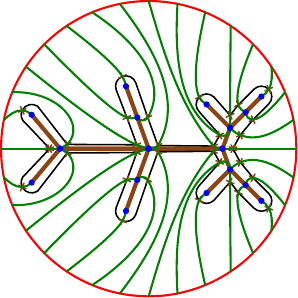}
  \caption{Planar tree graph of a generic vector field and 
  its attachement to the separatrices, where the vertices are blue,
  the egdes are brown and the separatrices are green. The tubular curve around
	the graph is the path used to attach the separatrices, described
  in Proposition~\ref{prop:tree-graph}.}
  \label{fig:tree graph}
\end{figure}

\begin{theorem}\label{thm:involution}\cite{DES05}  Let us consider $\mathcal{P}_{\Cc, k+1}$ for fixed $k$.
\begin{enumerate}
\item The generic strata are in bijection with the combinatorial parts and any combinatorial part is realizable. 
\item The separatrices at $\infty$ separate a neighborhood of infinity into $2k$ sectoral regions, called \emph{ends} at $\infty$. There is a unique involution linking ends two by two by non-crossing curves, which do not intersect the separatrices. This involution has no fixed point.
\item Any non-crossing involution without fixed points between the $2k$ ends is realizable by a combinatorial invariant.
\item The number of generic strata in $\mathcal{P}_{\Cc, k+1}$  is $C(k)$, where
$$C(k)= \frac{\binom{2k}{k}}{k+1}$$ is the $k$-th Catalan number. 
\end{enumerate} \end{theorem}

\begin{corollary}\label{rem:ends} The non crossing involutions without fixed
points of Theorem~\ref{thm:involution} are in bijection with the noncrossing
involutions without fixed points inside the unit circle linking 2 by 2  the
marked points 
\begin{equation} e_j =
      \exp\left({i\pi{2j-\sgn(j)\over 2k}}\right), \qquad j\in\{\pm1,\ldots,\pm k\}.\label{def:e}
\end{equation}
\end{corollary}

\begin{corollary}\label{rem:alternative} Alternative ways to describe the combinatorial parts are the following (\cite{DES05} or \cite{D13}):
\begin{enumerate} 
\item The set of combinatorial parts is in bijection with the set of paths in $\N^2$ from $(0,0)$ to $(2k,0)$ with steps $(1,1)$ and $(1,-1)$. Such a path is called a \emph{Dyck path of semilength $k$}. 
\item The set of combinatorial parts is in bijection with the set of admissible brackets for multiplication of $k+1$ numbers, namely the set of sequences $a_1, \dots, a_{2k}$, with $a_i\in\{(,)\}$, $a_i= ($ for half of the $a_i$,  and such that for each $j\in\{1, \dots, 2k\}$ $$
\# \{i \leq j, a_i=(\} \geq \# \{i \leq j, a_i=)\}.$$
\end{enumerate} \end{corollary} 

The attachment of the singular points can be recovered from the combinatorial invariant. This is described in the following proposition.

\begin{proposition}\label{attachment} \cite{DES05} Let $s_j$ be the separatrix of infinity tangent to the direction $\exp\left(\frac{\pi i j}{2k}\right)$, $j=0, \dots, 2k-1$. Let the permutation $\sigma$ be defined on $\{0,\ldots,2k-1\}$ by $$\sigma(j) := \tau(j+1)\quad ({\rm mod}\: 2k).$$ For each $j$, there exists a minimal $r_j>0$ such that $
\sigma^{\circ r_j}(j)= j$. This partitions $\{0,\ldots,2k-1\}$ into $k+1$ subsets of the form $\{j, \sigma(j), \dots, \sigma^{r_j-1}(j)\}$. Each subset is  the subset of the indices of the separatrices landing at one singular point. The subsets are the equivalence classes of the quotient $\{0,\ldots,2k-1\}/\sigma$. The equivalence classes are noted $[j]$. 
\end{proposition}

\subsection{The analytic part of the invariant of a structurally stable vector field}

Inside a given generic stratum of $\mathcal{P}_{\Cc, k+1}$, a vector field of $X_\eps\in \mathcal{P}_{\Cc, k+1}$ is characterized by the \emph{analytic part} of its invariant, which is defined as follows. 

\begin{theorem}\cite{DES05}\label{thm:temps transvers}
Let $X_\eps\in \mathcal{P}_{\Cc, k+1}$ be a generic vector field. 
\begin{enumerate}
\item The union of the separatrices, called the \emph{separatrix
graph}, divides $\Cc$ into $k$ open regions, called \emph{zones},
adherent to exactly two singular points, one attracting, one
repelling. Each zone is the union of all equivalent trajectories
joining its two adherent points. 

\item In the rectifying coordinate $t= \int\frac{dz}{P_\eps(z)}$,
each zone is biholomorphic to a horizontal strip.  Images of
$\infty$ (which can be reached in finite time) appear, one on the upper boundary, one on the
lower boundary of the strip. Let  $\eta\in \Hh$ be the vector joining the lower image of
$\infty$ to the upper one. Let $\theta\in(0,\pi)$ be the argument of
$\eta$. The image in $z$-space of the line segment $[0,\eta]$ in
$t$-space is a homoclinic loop of $Y_\eps= e^{-i\theta} X_\eps$.
This loop is included in the zone. In particular, it has empty
intersection with the separatrix graph and intersects one edge of
the planar tree graph. We call $\eta$ the \emph{transversal time} of
the zone. 
\item Using the marked points $e'_{j} = e^{i\pi {2j-1
\over 2k}}$, let $\eta_j = \eta_{(2j-1),\tau(2j-1)}$ for
$j\in\{1,\ldots,k\}$, where $\eta_{\ell,\tau(\ell)}$ is the transversal
time of the zone with ends at infinity  in the direction $e'_\ell$ and $e'_{\tau(\ell)}$ for odd $\ell$. 
The \emph{analytic part of the invariant} is the vector $\eta = 
(\eta_1,\ldots,\eta_k)\in \Hh^k$.
\item Any $X_\eps\in \mathcal{P}_{\Cc, k+1}$ is uniquely determined by the combinatorial and analytic parts of its invariant. 
   \end{enumerate} 
   Any stratum of generic vector fields is analytically parameterized by $\Hh^k$, namely the map $\mathbf{\eta}\mapsto \eps$ is holomorphic. In particular, any couple of a combinatorial part and an analytic part is realizable by a unique generic vector field in $\mathcal{P}_{\Cc, k+1}$.
\end{theorem}

\section{Generic vector fields in $\mathcal P_{\R,k+1}$}\label{sec:P_R}

\begin{definition} A vector field $X\in \mathcal P_{\R,k+1}$ is
  \emph{generic} if it is structurally stable.
\end{definition}

The following proposition is immediate. \begin{proposition} A vector field $X\in \mathcal P_{\R,k+1}$ is generic if and only if it has only simple singular points
  and $\R$ is the only possible homoclinic loop.\end{proposition}

\subsection{The combinatorial part of the invariant}

One combinatorial invariant is presented in \cite{BD10} in the general case of polynomial vector fields
with homoclinic loops and/or parabolic points. We will not need the full generality, so we present a simplified
version.

For a generic $X\in \mathcal P_{\R,k+1}$, when there are singular points on the real
axis, the strata are completely determined by their combinatorial type, i.e. a
non-crossing involution without fixed points $\tau$ inside the unit circle linking 2
by 2  the marked points $e_j$ defined in \eqref{def:e}  so that: 
\begin{itemize}
\item The links do not cross; 
\item A marked point $e_j$ can only be linked to a marked point $e_{j'}$ such that either $j-j'$ is odd and $jj'>0$, or $j'=-j$; 
\item If $e_j$ is linked to $e_{j'}$ with $jj'>0$, then $e_{-j}$ is linked to $e_{-j'}$. 
\end{itemize}
Hence it suffices to describe the linkings of points $e_{j}$ for $j>0$. We set
$\otau(j) := |\tau(j)|$ for $j>0$, which defines a non crossing involution of
$\{1,\ldots, k\}$. 

There is a slight difference when there is no singular point on the real axis: marked
points in each hemicycle can only be paired with marked points of the same
hemicycle. But in  that case $k$ is odd, and there is an odd number of marked
points. Hence no pairing using all marked points is possible and any pairing of
$k-1$ points linking $e_1$ to some $e_j$ ($j$ even), is equivalent to
the pairing where $e_k$ is linked to $e_j$.

\begin{definition}[Combinatorial invariant for generic $X\in \mathcal P_{\R,k+1}$]
  \label{def:invariant comb}
  When there are singular points on the real axis, the combinatorial invariant
  is the non crossing involution $\otau$ of $\{1,\ldots, k\}$ described above.
  When the real axis is a homoclinic loop, the combinatorial invariant
  is the non crossing involution $\otau$ of $\{1,\ldots,k-1\}$, with no fixed point described above. \end{definition}

 Recovering the attachment of the separatrices to the singular points is described in Proposition~\ref{attachment} when there is at least one real singular point. Hence we only need to describe the case when the real axis is a homoclinic loop.

\begin{proposition}\label{attachment_m=0}  Let $k$ be odd and consider the case where there are no real singular points. There are $\frac{k+1}2$ singular points and $k-1$ free separatrices in each of the upper and lower half-planes. Let $s_j$ be the separatrix of infinity tangent to the direction $\exp\left(\frac{\pi i j}{2k}\right)$, $j=1, \dots, k-1$ in the upper half-plane. Let the permutation $\ov{\sigma}$ be defined on $\{1,\ldots,k-1\}$ by $$\ov{\sigma}(j) := \otau(j+1)\quad ({\rm mod}\: k-1).$$ Each equivalence class of the quotient $\{1,\ldots,k-1\}/{\ov{\sigma}}$  is  the subset of the indices of the separatrices landing at one singular point. The equivalence classes are noted~$[j]$. 
\end{proposition}

Note that $j\mapsto -\otau(-j)$ allows us to partition $\{-1,\ldots,
-k+1\}$ in a similar way for the separatrices in the lower
half-plane. In particular, we have $[-j] = -[j]$.

\subsection{The analytic part of the invariant for a generic vector field of $\mathcal{P}_{\R, k+1}$}

\begin{proposition}\label{thm:R_temps transvers}
Let $X_\eps\in \mathcal{P}_{\R, k+1}$ be a generic vector field, and let $m$ be the number of real singular points. 
\begin{enumerate}
\item The transversal times of zones symmetric with respect to the $\R$-axis are elements of $i\R^+$ (these zones correspond to fixed point of $\otau$). There are $m-1$ such transversal times. 

\item The transversal times of pairs of zones symmetric to one another with
respect to the $\R$-axis are of the form $\eta_{j,\otau(j)},
-\ov{\eta_{j,\otau(j)}} \in \Hh$   for
$j,\otau(j)\in\{1, \dots, k\}$, $j$ odd. In the case $m=0$, then
$j,\otau(j)\in\{1, \dots, k-1\}$.

\item When $m=0$, the period of the $\R$-axis is an element of $\R^+$ that we call $\eta_\R$.

\end{enumerate}\end{proposition}

\begin{definition}
The analytic part of the invariant is the vector $\eta$, whose coordinates are the transversal times and the period of the $\R$-axis if applicable. The order of the coordinates is described later in Definition~\ref{def:order}. \end{definition}

\section{The generic strata of $\mathcal{P}_{\R,k+1}$}\label{sec:generic_strata}

We now specialize to vector fields of $\mathcal{P}_{\R,k+1}$. The main difference with $\mathcal{P}_{\Cc,k+1}$ is that the real axis is always invariant. And when there are no real singular points (an open condition), it is a homoclinic loop through $\infty$. 

Hence there are two types of generic strata in $\mathcal{P}_{\R,k+1}$, depending on the number $m$ of singular points on the real axis:
\begin{itemize} 
\item Strata of generic vector fields with singular points (of node
type) on the real axis ($m>0$). These vector fields are also generic inside $\mathcal{P}_{\Cc,k+1}$.
\item Strata of generic vector fields with $m=0$. Note that this case only occurs when $k$ is odd. Then the real axis is a homoclinic loop of $\infty$, and  such a generic vector field is not generic in $\mathcal{P}_{\Cc,k+1}$.
In each hemicycle there are $\frac{k+1}2$ singular points and $k-1$ separatrices. In the generic cases the singular points are linked as a planar tree graph with trajectories, and the tree is linked to the separatrices. This can be done in $C(\frac{k-1}2)$ ways as determined in Proposition~\ref{prop:tree-graph}. \end{itemize} 

\subsection{Number of generic strata}\label{sec:generic strata}

\begin{theorem}\label{thm:number-strata} The number $D(k)$ of generic strata of vector fields in $\mathcal P_{\R,k+1}$ is
$$D(k) =  2\binom{k-1}{\lfloor \frac{k-1}2\rfloor}= \begin{cases} \binom{k}{k/2},& k\:\text{even},\\
2\binom{k-1}{(k-1)/2}, & k\:\text{odd}.\end{cases}$$
For $k$ odd, the number of generic strata with no real singular points is $C\left(\frac{k-1}2\right)$, the $\frac{k-1}2$-th Catalan number.
\end{theorem}
\begin{proof} 
The result follows from this recurrence formula that will be proved geometrically:
$$D(k+1) = \begin{cases} 2D(k),& k\:\text{even},\\
2\left(D(k) - C\left(\frac{k-1}2\right)\right), & k\:\text{odd}.\end{cases}$$

\noindent(We will see that for $k$ odd, $C\left(\frac{k-1}2\right)$ is the number of strata with no real singular points.) 
\medskip

Since we consider generic strata, the vector fields have $k+1$ simple singular points and the real axis is invariant. Let $m$ be the number of singular points on the real axis. Then $k+1-m=2\ell$ is even and there are $\ell$ singular points in each hemicycle. \medskip

\medskip

\noindent{\bf Going from $k$ to $k+1$.} Let us introduce the term $k$-stratum to specify the value of $k$ for a given stratum.  Also the marked points depend on $k$, so let us define
\begin{equation} e_{j,k} =
      \exp\left({i\pi{2j-\sgn(j)\over 2k}}\right), \qquad j\in\{\pm1,\ldots,\pm k\}.\label{def:e,k}
\end{equation}
      From any generic $k$-stratum with $m\neq0$, we construct two different generic $(k+1)$-strata. All the existing generic $(k+1)$-strata 
are obtained in this way, and all the $(k+1)$-strata obtained in the construction are distinct. 
Let us now describe the construction. Let us start with a generic $k$-stratum with $m\neq0$. It has $m$ singular points on the real axis and symmetric branches (possibly void) starting from the real singular points. Hence each real singular point is the root of a tree (possibly with one vertex) in the upper hemicycle, and a root of the symmetric image of the tree in the lower hemicycle. 

Let $z_m< z_{m-1} < \dots < z_1$ be the real singular points for the $k$-stratum. The construction consists in adding a singular point $z_s$ on the real axis to the left of $z_m$ and two extra separatrices on both sides of the negative real axis. In practice, this is the $k+1$-(stratum) of  $ P(z) + \eta z^{k+2}$ with $\eta>0$ small. 
Before adding $z_s$, the negative infinite part of the real axis is a separatrix landing in $z_m$. After adding $z_s$, this separatrix splits into the two new separatrices which again land in $z_m$ and the negative infinite part of the real axis becomes a separatrix landing in $z_s$. 
This gives the first new generic $(k+1)$-stratum (see Figures~\ref{Induction1} and \ref{Induction2}). 

As for the second generic $(k+1)$-stratum, we let $z_s$ coallesce with $z_m$ and the two points move together as complex conjugates outside the real axis. For instance, $z_m$ moves in the upper hemicycle, and $z_s$ in the lower one. Then $z_m$ brings with it its attached tree of singular points, which remain attached to their separatrices.  The new separatrices created in the construction are attached to $z_m$ and $z_s$ respectively. All the other separatrices keep the same attachment. Two cases now occur. If $m=1$, then the new vector field has no real singular point and the real axis is a homoclinic loop (Figure~\ref{Induction2}). Otherwise, the negative infinite part of the real axis becomes a separatrix landing at $z_{m-1}$ (Figure~\ref{Induction1}).

\begin{figure}\begin{center} \includegraphics[width=12cm]{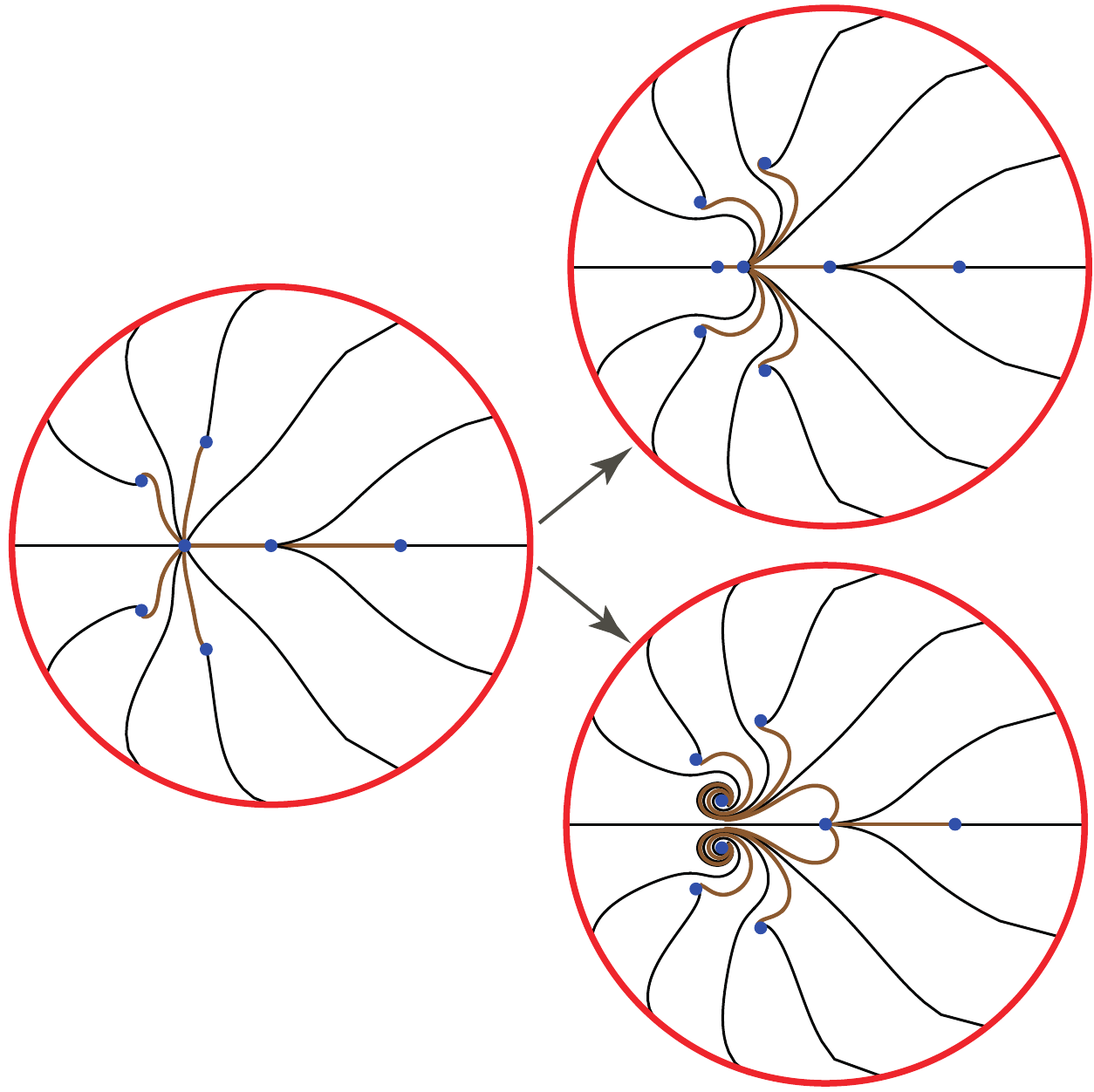}\caption{The two $(k+1)$-strata produced from a $k$-stratum when $m>1$.}\label{Induction1}\end{center}\end{figure}
\begin{figure}\begin{center} \includegraphics[width=12cm]{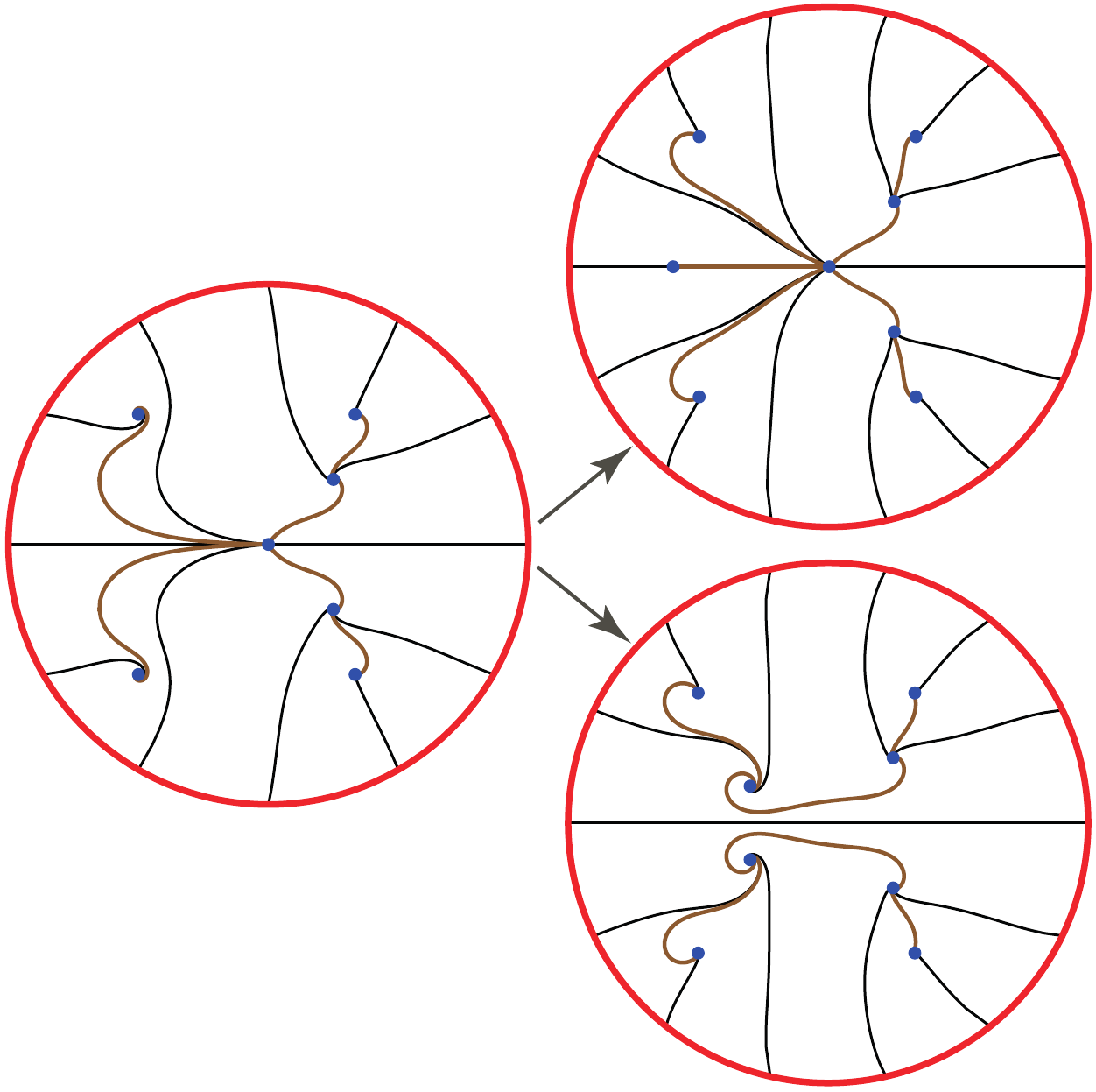}\caption{The two $k+1$-strata produced from a $k$-stratum when $m=1$.}\label{Induction2}\end{center}\end{figure}

\medskip\noindent{\bf Injectivity of the construction.} 

When adding the extra singular point $z_s$, the old marked points $e_{j,k}$ are moved to $e_{j,k+1}$ and two new marked points are added, namely $e_{\pm (k+1), k+1}$. If $z_s$ is on the real axis the new combinatorial invariant associates $e_{k+1, k+1}$ with $e_{- (k+1), k+1}$ and the rest of the linking is the same as before. 

If $z_m$ and $z_s$ are complex conjugates, and $m>1$, then the $k$-pairing for the $k$-stratum associates some $e_{j,k}$ to $e_{-j,k}$ with a link cutting the segment $[z_m,z_{m-1}]$. Then for the $(k+1)$-stratum, $e_{k+1,k+1}$ will be linked to $ e_{j,k+1}$ and the rest of the pairing  is identical.
If $z_m$ and $z_s$ are complex conjugates, and $m=1$, then the new pairing for the $(k+1)$-stratum is the same as before for all $e_{j,k+1}$, $j=1, \dots k$, and as discussed above, pairing $e_{k+1,k+1}$ is irrelevant. (See also Table~\ref{table1}.)

\begin{table} \begin{center} \begin{tabular} {|| l | c| l||}
\hline
\qquad Condition & $k$-pairing &$\qquad\qquad\quad (k+1)$-pairing\\
\hline
$z_s< z_m$ & $\tau_k$& $\begin{cases}  
\tau_{k+1}(j)=\tau_k(j), &1 \leq j \leq k\\
\tau_{k+1}(k+1)=- (k+1),\end{cases}$\\
\hline
$z_m=\ov{z}_s\in \Hh$, $m>1$& $\begin{cases} \tau_k(j) =-j\\  \tau_k(j') \neq-j',&j'>j\end{cases}$  &  $\begin{cases} \tau_{k+1}(j)=k+1\\ \tau_{k+1}(\ell)= \tau_k(\ell) , & \ell\neq j, k+1\end{cases}$ \\

\hline
$z_m=\ov{z}_s\in \Hh$, $m=1$&$\tau_k$&$\begin{cases} \tau_{k+1}(j)=\tau_k(j), &1 \leq j \leq k\\
\tau_{k+1}(k+1)\:\text{irrelevant}\end{cases}$\\
\hline \end{tabular}\end{center}\caption{Changes in the pairing when passing from a $k$-stratum with pairing $\tau_k$ to a $(k+1)$-stratum with pairing $\tau_{k+1}$.}\label{table1} \end{table}

Any pairing of a $(k+1)$-stratum can be uniquely obtained by this method from a pairing of a $k$-stratum with at least one real singular point and either adding $z_s$ on the real axis or as a complex conjugate of $z_m$. Indeed, if $m>0$ for the $(k+1)$-stratum, then its tree graph is symmetric with respect to the real axis. If this tree graph has one end on the leftmost side of the real axis then it can only come from a $k$-stratum by adding $z_s$ on the real axis. Otherwise, it comes from a $k$-stratum with $z_m$ and $z_s$ complex conjugates. 
\end{proof} 

\begin{corollary} For any value of $k$, half of the strata with real singular points have the leftmost singular point with no attached branches to it in the hemicycles.
Similarly, half of the strata with real singular points have the rightmost singular point with no branches in the hemicycles.\end{corollary}
\begin{proof} 
Indeed this is the case for $k=1$ and $k=2$. For any higher $k$ and for the leftmost point, it follows from the construction in the proof of Theorem~\ref{thm:number-strata}. The case of the rightmost case follows by symmetry using $z\mapsto -z$.  \end{proof} 

\begin{definition} A \emph{ dispersed Dyck path of length $n$}  is a path in $\N^2$ from $(0,0)$ to $(n,0)$ with steps $(1,1)$, $(1,-1)$ and $(1,0)$ and with no steps $(1,0)$ at positive height. \end{definition}

The number of dispersed Dyck paths of length $n$ is well known in the literature. We provide a proof for the sake of completeness. 

\begin{proposition} The number of dispersed Dyck paths of length $n$ is given by $\binom{n}{\lfloor\frac{n}2\rfloor}$.
\end{proposition}
\begin{proof} Let $N=\begin{cases} 0,& n\: \text{even},\\
-1,&n\: \text{odd}.\end{cases}$. Let us show that the number of dispersed Dyck paths of length $n$ is equal to the number of paths of length $n$ from $(0,0)$ to $(n,N)$ with steps $(1,1)$ and $(1,-1)$. This latter number is obviously equal to the number of choices for the steps $(1,1)$, namely $\binom{n}{\lfloor\frac{n}2\rfloor}$. The bijection is constructed as follows. 

Consider a dispersed Dyck path. Let the steps $(1,0)$ occur between $(n_i,0)$ and $(n_i +1,0)$, $i= 1, \dots, \ell$. A step $(1,0)$ occuring between $(n_i,0)$ and $(n_i +1,0)$ is replaced by a step $(1,-1)$ (resp. $(1,1)$) if $i$ is odd (resp. even). 
Moreover, for $i$ odd, the part $\gamma_i$ of the path between $(n_i +1,0)$ and $(n_{i+1},0)$ is replaced by $-\gamma_i -(0,1)$.

The inverse map is constructed as follows. Consider a  path of length $n$ from $(0,0)$ to $(n,N)$ with steps $(1,1)$ and $(1,-1)$. 
We change each descending step from level $0$, namely a step from some $(n_i,0)$ to $(n_i+1,-1)$, to a step from $(n_i,0)$ to $(n_i+1,0)$. Each ascending step to level $0$ from some $(n_i,-1)$ to $(n_i+1,0)$ is changed to a step from $(n_i,0)$ to $(n_i+1,0)$. And each portion $\delta_i$ of the path between $(n_i+1,-1)$ and $(n_j,-1)$, $n_j>n_i$ located below the line $y=-1$ is changed to  $-\delta_i-(0,1)$. \end{proof}

\begin{corollary} 
\begin{enumerate}
\item For $k$ even, the generic strata in $\mathcal{P}_{\R,k+1}$ are in bijection with the set of dispersed Dyck paths of length $k$, whose number is $Dd(k)= \binom{k}{k/2}$.
\item For $k$ odd, the strata generic in both $\mathcal{P}_{\R,k+1}$ and $\mathcal{P}_{\Cc,k+1}$ (i.e. at least one real fixed point) are in bijection with the set of dispersed Dyck paths of length $k$, whose number is $Dd(k)= \binom{k}{\lfloor k/2\rfloor}$. The number of generic strata with no real singular points is $C\left(\frac{k-1}2\right)$.
Hence $$D(k)= Dd(k) + C\left(\frac{k-1}2\right)= 2\binom{k-1}{(k-1)/2}.$$\end{enumerate}
\end{corollary}
\begin{proof} Recall that for even $k$ or for odd $k$ with
at least one real singular point, any generic vector field of $\mathcal{P}_{\R,k+1}$ is generic in $\mathcal{P}_{\Cc,k+1}$. 
Because of the symmetry, to describe a non-crossing involution $\tau$ between the marked points, it suffices to describe the images $\tau(j)$ for $j=1, \dots, k$ as steps of a dispersed Dyck path:
\begin{itemize}
\item The step is $(1,1)$ if $\tau(j) = j'$ with $j'>j$; \item The step is $(1,-1)$ if $\tau(j) = j'$ with $j'<j$; 
\item The step is $(1,0)$ if $\tau(j) = -j$.

\end{itemize}\end{proof} 

\subsection{Description of the strata} 

In this section, we count the number of generic strata with another approach, linked to their geometry. 

\begin{theorem}\label{Prop1} Let $D(k,m)$ be the number of generic strata with exactly $m$ real singular points. 
Then \begin{equation}\label{Dkm}
  D(k,m) =\begin{cases} \displaystyle \sum\limits_{\substack{n_1+\cdots+n_m = \frac{k-m+1}2\\
    n_j\geq 0}} C(n_1) \cdots C(n_m),& m>0,\\[8mm]
    C\bigl( \frac{k-1}2 \bigr), &m=0,\end{cases}
   \end{equation}
where $C(n)$ is the $n$-th Catalan number.
The number $D(k)$ of generic strata is   $$
    D(k) = \begin{cases} \displaystyle \sum_{m \in\{1, 3, \dots, k+1\}} D(k,m),& k\ \text{even},\\ \displaystyle \sum_{m \in\{0, 2, \dots, k+1\}} D(k,m),& k\ \text{odd}.\end{cases}$$
This yields the following combinatorial formulas
$$ D(k)= 
\begin{cases}    \displaystyle \sum_{m \in\{1, 3, \dots, k+1\}}
      \sum_{\substack{n_1 + \cdots + n_m = \frac{k-m+1}2\\
        n_j\geq 0}} C(n_1) \cdots C(n_m)
      = \binom k{k/2},&k\ \text{even},\\
      \displaystyle \sum_{m \in\{2, 4,\dots, k+1\}}
      \sum_{\substack{n_1 + \cdots + n_m = \frac{k-m+1}2\\
        n_j\geq 0}} C(n_1) \cdots 
				\textstyle C(n_m)+C\!\left( \frac{k-1}2 \right)\\
        \hskip19em   = 2\binom{k-1}{(k-1)/2}, &k\ \text{odd}.\end{cases}
  $$
\end{theorem}
\begin{proof} Let $m$
be the number of real singular points, which is odd (resp. even) when $k$ is
even (resp. odd). In all cases, $k-m+1$ is even, and there are ${k-m+1\over 2}$
singular points in each of the upper and lower half-planes.  As previously,
each real singular point is the root of two symmetric trees, one in the closed
upper half-plane, called \emph{upper tree}, and its mirror image in the closed lower
half-plane.  Because of the symmetry, we only discuss the upper half-plane.

Let us first discuss $m>0$, and let $z_m< \ldots < z_1$ be the real singular
points and $n_j+1$ be the number of vertices of the upper tree rooted at $z_j$,
so $n_j$ is the number of edges of the upper tree and also the number of non
real vertices. Then we have the relation $$ \sum_{j=1}^m n_j = {k-m+1\over 2}.
$$ For $n_1,\ldots,n_m$ fixed, then $C(n_1)\cdots C(n_m)$ is the number of
possible ``ordered'' forests of $m$ trees with the
$j$-th tree having $n_j$ edges, where $C(n) = \frac1{n+1}{2n \choose n}$ is
the $n$-th Catalan number, which counts the number of rooted trees with $n$
edges. Therefore, the number of forests with $m$ rooted trees with a total
number of $\frac{k-m+1}2$ edges is $D(k,m)$ given in \eqref{Dkm}.

In the case $m=0$, there are $\frac{k+1}2$ singular points
($\frac{k-1}2$ edges) and $k-1$ separatrices in the upper half-plane that form a  tree. Quotienting the real
axis to a regular point, this is completely similar to the DES case
for $k'= \frac{k-1}2$ and the number of configurations of a tree
attached to the separatrices is $  D(k,0) = \textstyle C\bigl(
k'\bigr)= \textstyle C\bigl( \frac{k-1}2 \bigr).  $ \end{proof}

\section{Structure of the family $\mathcal P_{\R,k+1}$}\label{sec:structure}

The bifurcations occuring in the family $\mathcal P_{\R,k+1}$ are of the following  types:
\begin{enumerate}
\item Bifurcations of homoclinic loops through $\infty$. The homoclinic loop along the real axis occurs on an open set in parameter space while the other bifurcations of homoclinic loops occur on algebraic sets of real codimension~1. 
\item Bifurcation of parabolic singular points. The bifurcation of  multiple singular points of multiplicity $n+1$ occurs in a set of real codimension $n$ when the singular point is on the real axis and of complex codimension $n$ otherwise. We will show below that full unfoldings exist inside the family $\mathcal P_{\R,k+1}$. \item Intersection of the former types.\end{enumerate} 

The family $\mathcal P_{\R,k+1}$ of polynomial vector fields with real coefficients
$P_\eps(z) \frac{\partial}{\partial z}$ for 
$$P_\eps(z) = z^{k+1} + \eps_{k-1} z^{k-1} + \dots + \eps_1z+\eps_0$$ 
with $\eps\in \R^k$ is as complete as possible under the constraint that the real axis is invariant. 

\begin{theorem} Let $k \geq2$ and $z_0$ be a non real singular point of $P_{\eps^*}(z) \frac{\partial}{\partial z}$ for a particular value $\eps^*$ of $\eps$. Then  a full unfolding exists in the family $\mathcal P_{\R,k+1}$.\end{theorem}
\begin{proof} Let $s$ be the multiplicity of $z_0$: $z_0$ is simple if $s=1$ and multiple otherwise. We can of course suppose that $z_0=a+i$, with $a\in \R$. 
Then $P_{\eps^*}(z) = (z-a-i)^s(z-a+i)^sQ_{\eps^*}(z)$, where ${\rm deg}(Q_{\eps^*})= \ell=k+1-2s$ and $Q_{\eps^*}$ is a polynomial with real coefficients. 

If $\ell>0$, let us consider the following unfolding
$$P_{\eps(\nu)}(z) = R_{\eps(\nu)}(z) \ov{R}_{\eps(\nu)}(z) Q_{\eps(\nu)}(z),$$
with $$\begin{cases}
R_{\eps(\nu)}(z)=(z-a-i)^s + \sum_{j=0}^{s-1} (\delta_j+i\eta_j) (z-a-i)^j,
  \\[2\jot]
\ov{R}_{\eps(\nu)}(z)=(z-a+i)^s + \sum_{j=0}^{s-1}(\delta_j-i\eta_j) 
  (z-a+i)^j,\\[2\jot]
Q_{\eps(\nu)}(z)= Q_{\eps^*}(z)-2\delta_{s-1}z^{\ell-1} +\sum_{j=0}^{l-2} \mu_jz^j,\end{cases}$$
where $\nu=(\delta_0, \eta_0,\dots, \delta_{s-1}, \eta_{s-1}, \mu_0, \dots \mu_{\ell-2})\in \R^k$. Let $\eps(\nu)=\eps^*+\nu$.
Then $P_{\eps(\nu)}\in \mathcal P_{\R,k+1}$ and, from its form,  $P_{\eps(\nu)}$ is a complete unfolding of $P_{\eps^*}$ in the neighborhood of $a+i$, since $\ov{R}_{\eps(\nu)}(z) Q_{\eps(\nu)}(z) = C+ O(\nu) +O(z-a-i)$, i.e. $R_{\eps(\nu)}(z)$ is the Weierstrass polynomial associated to $a+i$. 
By symmetry, $P_{\eps(\nu)}$ is a complete unfolding of $P_{\eps^*}$ in the neighborhood of $a-i$. Using induction on the degree of $P_{\eps^*}$, we can also conclude that $P_{\eps(\nu)}$ is a complete unfolding of $P_{\eps^*}$ in the neighborhood of each zero of $Q_{\eps^*}$.

Let us now consider the case $\ell=0$. In this case, necessarily $a=0$, and  we consider the unfolding $P_{\eps(\nu)}(z) = R_{\eps(\nu)}(z) \ov{R}_{\eps(\nu)}(z)$ with
$$\begin{cases}
R_{\eps(\nu)}(z)=(z-i)^s + \sum_{j=0}^{s-2} (\delta_j+i\eta_j)(z-i)^j +i\eta_{s-1}(z-i)^{s-1},\\
\ov{R}_{\eps(\nu)}(z)=(z+i)^s + \sum_{j=0}^{s-2} (\delta_j-i\eta_j)(z+i)^j-i\eta_{s-1}(z-i)^{s-1},\end{cases}$$ 
where $\nu=(\delta_0, \eta_0,\dots, \delta_{s-2}, \eta_{s-2}, \eta_{s-1})\in \R^k$.
Then $P_{\eps(\nu)}$ is a complete unfolding of $P_{\eps^*}$ in the neighborhood of $\pm i$.
\end{proof}

\begin{remark} Consider $P_\eps$ generic. Since the eigenvalues of real singular points are real and the eigenvalues of complex conjugate points are complex conjugate, then $k+1$ real parameters would be needed if the eigenvalues were to be independent. But  the eigenvalues of the singular points $z_1, \dots, z_{k+1}$ of $P_\eps$ satisfy $\sum_{j=1}^{k+1} \frac1{P_\eps'(z_j)}=0$. Since this sum is real, this is a condition of real codimension 1. This explains why a generic $P_\eps$ depends on $k$ real parameters.\end{remark} 

\section{The case $k=3$}\label{sec:k=3} 
\begin{theorem} The bifurcation diagram of the vector field $(z^4+\eps_2z^2+\eps_1z+\eps_0)\frac{\partial}{\partial z}$ is given in Figure~\ref{Bif_diag-4}.
The bifurcation diagram has a conic structure using \eqref{scaling}. Its intersection with a sphere minus a point in the regular stratum with no real singular point is given in Figure~\ref{Bif_diag-4_2D}. \end{theorem}
\begin{figure} \begin{center} \includegraphics[width=15cm]{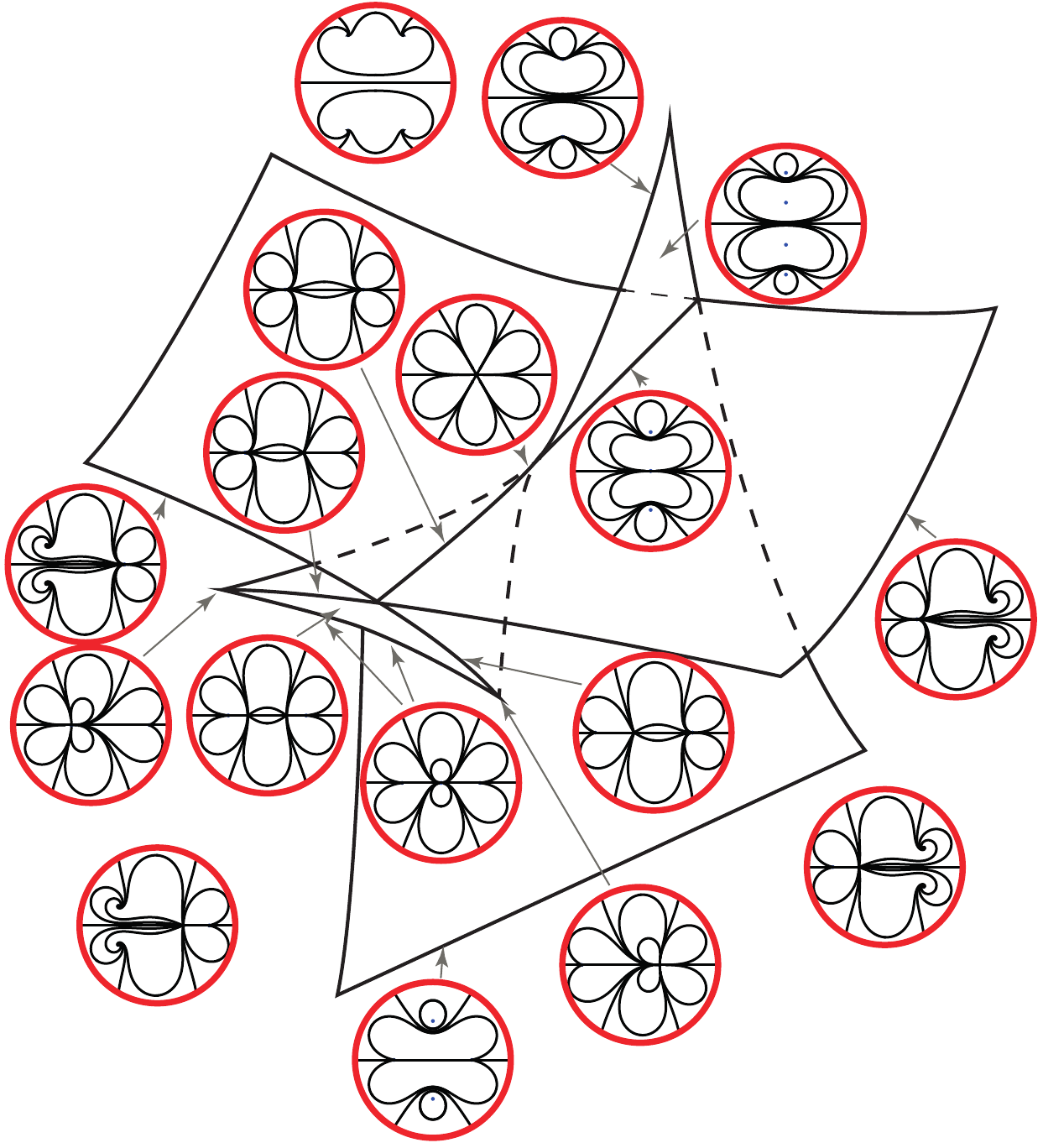}\caption{Bifurcation diagram of  $(z^4+\eps_2z^2+\eps_1z+\eps_0)\frac{\partial}{\partial z}$.}\label{Bif_diag-4}\end{center}\end{figure}
\begin{figure} \begin{center} 
\includegraphics[width=15cm]{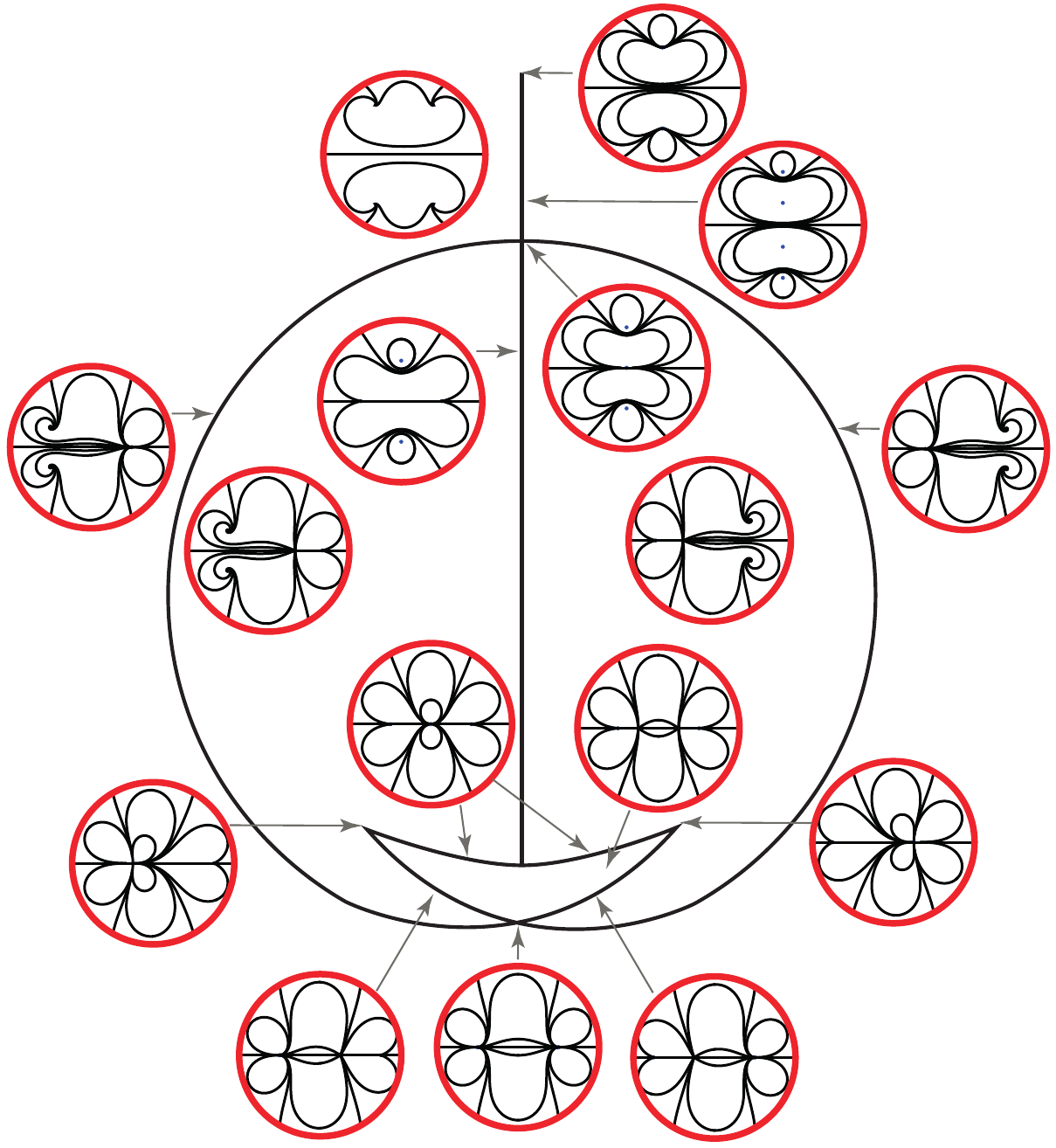}
\caption{The intersection of the bifurcation diagram of  $(z^4+\eps_2z^2+\eps_1z+\eps_0)\frac{\partial}{\partial z}$ with a sphere minus a point.}\label{Bif_diag-4_2D}\end{center}\end{figure}

\begin{proof} 
The codimension 1 bifurcations are of two types: real parabolic points of multiplicity 2 and pairs of homoclinic loops symmetric with respect to the real axis. 
The higher order bifurcations occur at the multiple points, located either  at the boundary of the codimension 1 bifurcation surfaces, or at the intersection of the codimension 1 bifurcation surfaces. Note that a complex double parabolic point has real codimension 2. 
\smallskip
 
\noindent{\bf Bifurcations of parabolic points.} They occur when the discriminant $\Delta= 256\eps_0^3-27\eps_1^4+144\eps_0\eps_1^2\eps_2-128\eps_0^2\eps_2^2-4\eps_1^2\eps_2^3+16\eps_0\eps_2^4$ vanishes. The locus where $\Delta=0$ corresponding to multiple real roots is well known: it is the swallow tail. There is also the locus of nonreal pairs of parabolic points $(z_0,\ov{z}_0)$.
Because their sum vanishes, then $z_0=ia$ for some $a\in\R^*$. This occurs along the curve $\eps_1=0$, $\eps_2^2-4\eps_0=0$ and $\eps_2>0$.  (Note that the other part of the curve $\eps_1=0$, $\eps_2^2-4\eps_0=0$, $\eps_2<0$, corresponds to a pair of real parabolic points: this is the self-intersection curve of the swallow tail.) 

\smallskip
\noindent{\bf Bifurcations of homoclinic loops}. Since the real axis is invariant, it is a homoclinic loop as soon as there are no real singular point. This occurs in an open region in parameter space, and it is not a bifurcation. Otherwise, homoclinic loops occur in symmetric pairs  when one singular point has a pure imaginary eigenvalue. There are two cases to consider: $m=0$ and $m=2$. 

When $m=2$, then $P$ has the form $P(z) = (z-2a)(z-2b)((z+a+b)^2+c^2)$ with $c\in\R^*$. $P'(-a-b\pm ic)$ is pure imaginary only if $(a+b)c^2=0$. Hence $a+b=0$,  $P(z) = (z-2a)(z+2a)(z^2+c^2)$ and $P'(\pm ic)= \mp2ic(4a^2+c^2)\neq0$. This occurs on $\eps_0 < 0$, $\eps_1 = 0$. 
(Note that necessarily $\eps_0<0$ when $\eps_1=0$ and $m=2$.) When $\eps_0=0$, this is the particular case of a parabolic real point at $z=0$. 

When $m=0$, then $P$ has the form $P(z)= ((z-a)^2+b^2)((z+a)^2+c^2)$ for
$b,c\in\R^*$. Then $P'(a+ib)\in i\R^*$ if and only if $a=0$ and $b\neq0,\pm c$.
This occurs for $\eps_1=0$, $\eps_2^2-4\eps_0>0$ and $\eps_0,\eps_2>0$. This
surface is bounded by $\eps_1=\eps_0=0$, $\eps_2>0$ where a real  parabolic
point occurs at $z=0$ and homoclinic loops around
$\pm i\sqrt{\eps_2}$, and $\eps_1=0$, $\eps_2^2-4\eps_0=0$, $\eps_2>0$ where a
pair of symmetric purely imaginary parabolic points occurs 
(and no homoclinic loops).

\end{proof}

\section{Realization}\label{sec:real}
\def\kk{$k$}

We briefly discussed the invariants of generic vector fields of $\mathcal P_{\R,k+1}$
in Section~\ref{sec:P_R}. In this section, we precisely define
their modulus of classification and we prove that any modulus is
realizable as a generic vector field of $\mathcal P_{\R,k+1}$. 

\subsection{Properties of the combinatorial and analytic invariants}

For a generic vector field of $\mathcal P_{\R,k+1}$, recall that the
combinatorial invariant is a non crossing involution (see
Definition~\ref{def:invariant comb})  and the analytic invariants are the
transversal times and potentially the period of a homoclinic loop along $\R$.
From Propositions~\ref{attachment} and \ref{attachment_m=0} we can recover how
the separatrices attach to the singular points from the combinatorial data. For
the rest of the section, we use the notation $z_{[j]}$ for a singular point,
where $[j]$ is an equivalence class of separatrices landing at the same singular point.
The combinatorial data tells us how many singular points are
real, as seen in Lemma~\ref{lem:module} below.  Indeed, $\ol{z_{[j]}} =
z_{[-j]}$, hence $z_{[j]}$ is real if and only if $[j] = [-j]$.

Recall that the separatrix graph of $X$ divides $\Cc$ into simply
connected zones. We denote a zone with two ends in the direction
$e_j,e_k$ by $Z_{j,k}$ (or $Z_{k,j}$ interchangeably), and the
transversal time in $Z_{j,k}$ by $\eta_{j,\tau(j)}$.

The following lemma explains how to recover some features of the dynamics, including the presence of a homoclinic loop along the real axis and the number of real singular points from the combinatorial data.

\begin{lemma}\label{lem:module}
Let $X = P\frac\del{d z} \in \mathcal P_{\R,k+1}$ be a generic
vector field. Let $\otau$ be the combinatorial
invariant of $X$: its domain $\dom(\otau)$ is a set of cardinality $k-1$ (resp. $k$) if $\R$ is a homoclinic loop (resp. otherwise). Let $\ell = k -\#\dom(\otau)$. (Then $\ell=1$ if $\R$ is a
homoclinic loop, and $\ell=0$ otherwise.)
\begin{enumerate}
  \item $\otau$ is equivalent to the non crossing involution with no fixed points, 
    $\tau$, of $\{\pm1,\ldots,\pm(k-\ell)\}$ that pairs the ends of zones of $X$.
  \item The number $m\in \N$ of real singular points of $X$ is given by
    \begin{align*}
      m+\ell-1
        &= \text{number of fixed points of $\otau$,} \\
        &= \text{number of pairs $\left(j,\tau(j)\right)$ with $j > 0$ and
          $\tau(j) < 0$,} \\
        &=\text{number of zones whose interior intersects $\R$.} 
    \end{align*}
  \item The transversal times satisfy
    $\eta_{j,\tau(j)} = -\overline{\eta_{-j,\tau(-j)}}$.
		In particulier,	there are $m+\ell-1$ transversal times in $i\R_+$.
\end{enumerate}
\end{lemma}
\begin{proof}
  For (1), $\tau$ and $j\mapsto -j$ commute since $P$ commutes with the
  complex conjugation. Then, we have 
	$$
		\tau(j) = 
			\begin{cases} 
				\otau(j), & \otau(j)\neq j, \\ 
				-j, &\otau(j)=j,
			\end{cases}
		\rlap{\qquad for $j>0$,}
	$$ 
	and $\tau(-j)=-\tau(j)$ for $j < 0$.

  For (2), if $\ell = 1$, then $\otau$ has no fixed points, so $m = 0 $ is the
  number of fixed points of $\otau$. For the case $\ell= 0$, let $n$ be the
  number of zones with an interior intersecting the real axis. This divides the
  real axis in $n$ finite intervals, each of which is an orbit of $P\frac{\partial}{\partial z}$ and, in
  addition, the two separatrices, along $\R_+$ and $\R_-$ on the boundary of some zones.
  This partitions $\R$ into $n+2$ orbits and $n+1$ singular points, one between
  each orbit. We conclude that $m=n+1$.

  (3) is a direct computation.
\end{proof}

\begin{notation}
  We will use the notation $\Fix(\otau) := \{j\mid \otau(j) = j\}$.
\end{notation}

For $m=0$, we have ${k-1\over 2}$ zones in the upper
half-plane. We can define a transversal time for each of them.
We also define the period of the $\R$-axis
$$
  \eta_\R = \int_\R \frac{{\rm d}z}{P(z)},
$$
which is an invariant. Note that $\eta_\R > 0$ since we only
consider the case where the leading coefficient of $P$ is
positive.

\begin{definition}[Modulus of classification of a generic vector field]\label{def:order}
  Let $P\frac\del{\del z}$ be a generic vector
	field of $\mathcal P_{\R,k+1}$. Let $m$ be the number of real
	roots of $P$, and $\ell$ be the number of homoclinic loops (either 0 or 1).
  The modulus of invariants of $P$ is a triple
	$(\otau,\eta)$, where 
  $\otau$ is a non crossing involution of $\{1,\ldots, k-\ell\}$
  with $m-1+\ell$ fixed points, and $\eta\in \Hh^{{k-m+1\over 2}-\ell}
  \times (i\R^+)^{m-1+\ell}\times (\R^+)^\ell$ is the vector whose coordinates are the 
  transversal times and period of $\R$, ordered as follows (see Figure~\ref{fig:tau}): 
  \begin{enumerate}
    \item 
let $0 < q_1 < \cdots < q_{(k-m+1)/2-\ell}$ be the odd numbers in $\{1, \dots, k-\ell\}$ which are not fixed points of $\otau$; then  we set $\eta_j$ as the transversal time for the zone
      $Z_{q_j, \otau(q_j)}$; 
    \item next, we consider $0 < r_1 < \cdots
      < r_{m-1}$ in $\Fix(\otau)$, and we set $\eta_{j+(k-m+1)/2 -\ell}$
      as the transversal time in $Z_{r_j,-r_j}$; 
    \item lastly, if $\ell=1$,
      the last coordinate is $\eta_\R = \int_\R \frac{dz}{P(z)}\in \R^+$. 
  \end{enumerate}
\end{definition}

\begin{figure}[tbph]
	\centering
	\subfigure[Zones with $k=7$, $m=4$ and $\ell=0$.\label{fig:tau a}]
	{\includegraphics{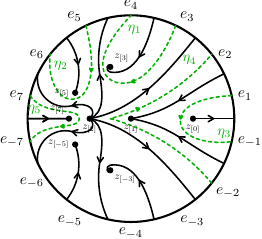}}
	\qquad\quad
	\subfigure[Zones with $k=7$, $m=0$ and $\ell=1$.\label{fig:tau b}]
		{\includegraphics{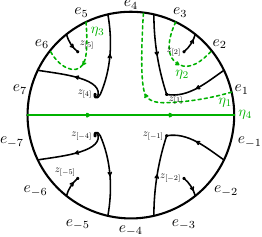}}
	\caption{Zones with marked ends $e_j = 
    \exp\left( i\pi{2j-\sgn(j)\over 2k}\right)$. 
    Transversal times of zones in symmetric pairs are 
    ordered first, then of zones crossing the real line (if applicable), then the period
    of the homoclinic loop (if applicable).}\label{fig:tau}
\end{figure}

\subsection{Realisation} This is the main result of the section.

\begin{theorem}[Realisation and moduli space]\label{theo:realize}
  Let $\mathcal S_k$ be the set of non
  crossing involutions of $\{1,\ldots,k\}$, and $\mathcal
  S_{k-1}^*$ be the set of the non crossing involutions of $\{1, \dots, k-1\}$
  without fixed point. (Note that $\mathcal S_{k-1}^*$ is empty when 
	$k$ is  even.)
  For $\otau \in \mathcal S_k\cup \mathcal S_{k-1}^*$, let
  $\ell = k - \#\dom(\otau)$ and $m = \#\Fix(\otau) + 1 - \ell$.
	The moduli space $\mathcal M_k$ of the
  generic vector fields in $\mathcal P_{\R,k+1}$ is
  $$
    \mathcal M_k = \bigcup_{\otau\in \mathcal S_k\cup \mathcal S_{k-1}^*}
      \{\otau\} \times \Hh^{{k-m+1\over 2} -\ell} \times
      (i\R^+)^{m-1+\ell}\times (\R^+)^\ell.
  $$
  For each $(\otau,\eta)\in \mathcal M_k$, there exists a unique
  monic centered generic polynomial $P$ of degree $k+1$ with real 
	coefficients such that $P\frac\del{\del z}$ has $(\otau,\eta)$ for
  its modulus, $m$ real singular points and $\ell$ homoclinic
  loop.
\end{theorem}

\begin{remarks} 
	\begin{enumerate}
	\item The theorem is stated for polynomials with a leading 
		coefficient of $+1$. When $k$ is odd, the polynomials
		$P(z) = z^{k+1} + \eps_{k-1}z^{k-1} + \cdots + \eps_0$
		and $Q(z) = -z^{k+1} - \eps_{k-1}z^{k-1} - \cdots - \eps_0$
		have the same invariants $(\overline\tau,\eta)$, but they
		are not equivalent. Indeed, they have the same orbits, but with
		reversed orientation. We could add a sign to the modulus for
		$k$ odd, but since we are only concerned with polynomials
		with a positive leading coefficient, we omit this addition.
	\item In all cases, the total number of real parameters is $k$ as expected since a generic stratum is an open set of $\R^k$.	\end{enumerate}
\end{remarks}

The proof has two main steps: \begin{itemize} 
\item 1.~the construction of a Riemann
surface $M$ conformally equivalent to a sphere punctured in $k+1$ points,
on which the vector field $1\frac\del{\del Z}$ is
well-defined and with a distinguished point at infinity; the polynomial vector field is then obtained as
the pushforward of $1 \frac\del{\del z}$ on $\Cc$  punctured in 
$k+1$ points and holomorphically extended with singular points 
at the punctures; 
\item 2.~defining a complex conjugation on $M$ to prove that the 
polynomial vector field has real coefficients.
\end{itemize} 
The construction of the Riemann surface is done
in~\cite{DES05} in the DES-generic case, and in \cite{BD10} when
there are homoclonic loops. We will briefly go over the parts
of the construction needed for the proof of our theorem 
so that the paper be self-contained.

\paragraph{Construction of a rectifying Riemann surface.} 

Let $\ell = k - \#\dom(\otau)$ and $m = \#\Fix(\otau) + 1 -\ell$.

We consider $k-\ell$ horizontal strips grouped in
two types. (See Figure~\ref{fig:zone strip DES b} for $\ell=0$ and
Figure~\ref{fig:zone strip non reel b} for $\ell = 1$.) 
\begin{enumerate}
	\item For each $j \in \Fix(\otau)$, set
	$$
		B_{j,\tau(j)} := \bigl\{ z\in \Cc : 
		\textstyle |\Im z| < \frac{\Im\eta_{j,\tau(j)}}2 \bigr\}
	$$
	and $B_{-j,\tau(-j)} := B_{j,\tau(j)}$. We 
	define the points $E_j := (-1)^j \frac{\eta_{j,\tau(j)}}2$
	and $E_{\tau(j)} = -E_j$ on the boundary of $B_{j,\tau(j)}$.
	\item For each $j\in \{1,\ldots,k-\ell\}\setminus \Fix(\otau)$, set
	$$
		B_{j,\tau(j)} = \bigl\{ z\in \Cc : 
		\textstyle 1 < \Im z < \Im\eta_{j,\tau(j)} + 1 \bigr\}
	$$
	and $B_{-j,\tau(-j)} := \overline{B_{j,\tau(j)}}$ (the complex
	conjugate of $B_{j,\tau(j)}$). We define the point $E_j = i$
	for $j$ odd or $E_j = \eta_{j,\tau(j)} + i$ for $j$ even,
	then we define $E_{\tau(j)} = E_j + (-1)^{j+1}\eta_{j,\tau(j)}$.
	Lastly, set $E_{-j} = \overline{E_j}$ and $E_{\tau(-j)} = 
	\overline{E_{\tau(j)}}$. In the case $\ell=1$, we set
  $E_k = E_1 - \eta_{\R}$ in $B_{1,\tau(1)}$ and $E_{-k}
  = \overline{E_k}$ in $B_{-1,\tau(-1)}$.
\end{enumerate}
We then have $m-1+\ell$ symmetric strips and $\frac{k-m+1}2-\ell$ pairs
of symmetric strips, for a total of $k -\ell$ horizontal strips,
and $2k$ marked points on the boundaries of the strips. 
On each strip, the vector field $1\frac\del{\del Z}$ is
well-defined. The first type of strips will correspond
to zones intersecting the real axis and the real line
in the strip will be mapped on an interval of the real axis.
The other type of strips will correspond to pairs of zones symmetric
with respect to the real axis.


Let us consider a zone not containing the real axis as a homoclinic
loop in its boundary. For $j>0$ and odd, the marked point $E_j$
divides the lower boundary of the strip in half-lines $S_{j-1}$ on
the left and $S_j$ on the right, and the marked point $E_{\tau(j)}$
divides the upper boundary of the strip in half-lines $S_{\tau(j)}$
on the left and $S_{\tau(j)-1}$ on the right if $\tau(j)>0$ (resp.
half-lines $S_{\tau(j)+1}$ on the left and $S_{\tau(j)}$ on the
right if $\tau(j)<0$) as in Figures~\ref{fig:zone strip DES b}
and~\ref{fig:zone strip non reel b}. For $j$ even and $\otau(j)=j$,
the marked point $E_j$ divides the upper boundary of the strip in
half-lines $S_{j}$ on the left and $S_{j-1}$ on the right, and the
marked point $E_{\tau(j)}=E_{-j}$ divides the lower boundary of the
strip in half-lines $S_{-j}$ on the left and $S_{-j+1}$ on the right
(Figure~\ref{fig:zone strip DES}).  For $j< 0$, we define $S_j :=
\overline{S_{-j}}$.  (The ordering of the $S_j$'s is completely
determined by $\tau$, but since $\tau(-j) = -\tau(j)$, this is
equivalent.)

We obtain a Riemann surface $M$ by identifying the strips along a
half-line $S_\ell$ on some $B_{j,\tau(j)}$ with a half-line $S_\ell$
on some $B_{n,\tau(n)}$.  Recall from
Propositions~\ref{attachment} and \ref{attachment_m=0} that the
combinatorial invariant partitions $\{\pm1,\ldots,\pm (k-\ell)\}$
in classes $[j]$. Now, if we glue every half
strip along the separatrices $S_{j'}$ on $S_{j'}$ for every
$j'\in[j]$, we obtain a half infinite cylinder conformally
equivalent to an open set $U$ with the topological type of a
punctured disk. 
This allows us to add a point $Z_{[j]}$ at the
unbounded end of the cylinder. Since there are $k+1$ equivalence
classes, there are $k+1$ half cylinders, so that we can add $k+1$
points to $M$ in this way.

Lastly, we identify $\{E_j\}_j$ to a single point $E$ and we define
a neighbourhood $V(E)$ of $E$ by taking a small half disk around
each $E_j$ inside $B_{j,\tau(j)}\cup \del B_{j,\tau(j)}$ and by
gluing them together along segments of separatrices. This is a
$k$-covering of $E$ and  the local  chart $ h\colon
\bigl(V(E),E\bigr) \to (\hat \Cc,\infty)$, given by $\zeta = {C\over
(Z-E)^{1/k}}$ sends $1\frac\del{\del Z}$ to
$\zeta^{k+1}\frac\del{\del \zeta}$, for some $C$ satisfying
$kC^k=-1$. There are $k$ such charts: we choose the one for which
the image of $S_0$ is oriented along the positive real axis and we
call it~$h_E$.

Let $\widetilde M = M \cup \{Z_{[j]}\}_j\cup \{E\}$ be the Riemann
surface thus obtained. It it proved in~\cite{DES05} ($m > 0$) or
in~\cite{BD10} ($m=0$) that $\widetilde M$ 
is homeomorphic to a sphere, and, following from
Riemann's Uniformization Theorem, isomorphic to the Riemann
sphere.

\begin{figure}[htbp]
  \centering
  \centerline{\subfigure[Zones in the $z$ coordinate (for $k=7$). The separatrix
  graph is in gray. The marked points $e_j$ represent the infinity
  in the direction $\exp\!\left(\pi i \frac{2j -1}{2k}\right)$.
  \label{fig:zone strip DES a}]
  {\includegraphics[scale=.9]{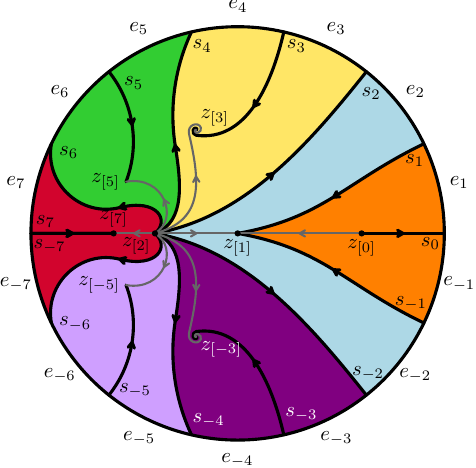}}
  \qquad
  \subfigure[Horizontal strips in the rectifying coordinate
  $t$. The coordinate is ramified at $z_{[j]}
  $ and $z_{[j]}$ corresponds
  to $Z_{[j]}$ ``at infinity''.
  The marked point $e_j$ (infinity in the $z$ coordinate) is mapped
  on $E_j$ on the boundary of a strip. The dashed lines are the images of the
  transversal curves.\label{fig:zone strip DES b}]
    {\includegraphics[scale=.9]{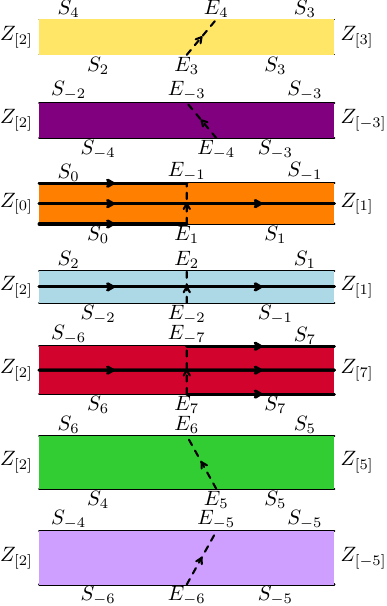}}}
  \caption{Representation of the zones and the strips in codimension~$7$ with
    at least one singular point on the real axis. A zone of a color
    is mapped on the strip of the same color. }
  \label{fig:zone strip DES}
\end{figure}

\begin{figure}[htbp]
  \centering
  \centerline{\subfigure[Strips in the $z$ coordinate.]{\includegraphics{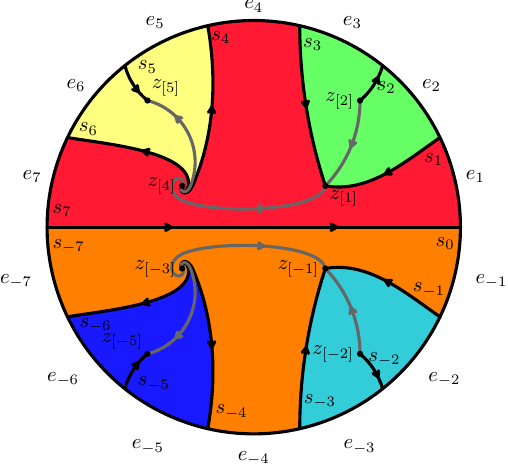}}
  \qquad
  \subfigure[Strips in the rectifying coordinate $Z$.\label{fig:zone strip
		non reel b}]
    {\includegraphics{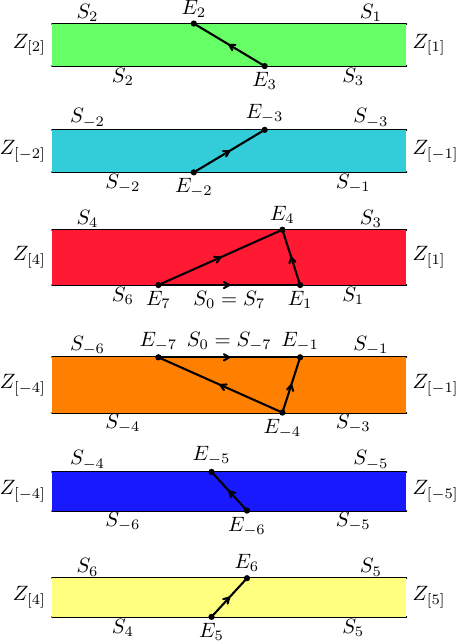}}}
  \caption{Representation of the strips in codimension~$7$ with
    no singular point on the real axis. A zone of a color is 
    mapped on the strip of the same color. The singular points $z_{[j]}$ are
    mapped on the $Z_{[j]}$ ``at infinity''.}
	\label{fig:zone strip non reel}
\end{figure}

\paragraph{Existence of the monic polynomial vector field}

There exists a unique isomorphism $H\colon \widetilde
M \to \hat\Cc$ such that $H$ in the chart
$\bigl(V(E),h_E\bigr)$ has the form $\zeta\mapsto\zeta +
o\bigl({1\over \zeta}\bigr)$. The vector field $1\frac\del{\del Z}$
is well defined on $M$, so we define the pushforward
$P\frac\del{\del z} := H_\ast\bigl(1\frac\del{\del
Z}\bigr)$. Since $P$ is bounded near each $z_{[j]}$, it can be
extended over each $z_{[j]}$ to a finite value, which
we need to prove to be zero. A chart on the half-cylinder
associated to $z_{[j]}$ compatible with the glueing of
points $E_n$ is given by $\xi =\exp(\lambda_j Z)$ where
$\lambda_j$ is the sum of the transversal times of the infinite
half-strips glued to form the half-cylinder. The pushforward in the
$\xi$-coordinate shows that the vector field has a
simple zero at $z_{[j]}$, a property invariant by change of
coordinate. Hence, $P(z_{[j]}) = 0$.  In a neighbourhood of
infinity, we have $(h_E\circ H\inv)_*\bigl(P\frac
\del{\del z}\bigr)(\zeta) = \zeta^{k+1} \frac\del{\del \zeta}$.
It follows that $P$ is a monic polynomial of degree $k+1$.

\begin{remark}\label{rem:coef +1} 
	We choose $h$ so that it maps $S_0$ to $s_0$ tangent to
  $\R_+$ at infinity. With this choice, it then follows that
  the polynomial obtained has a positive leading coefficient
  because of the way we attached the separatrices on the
  Riemann surface. 
\end{remark}

\paragraph{Proof of Theorem~\ref{theo:realize}}

  We define a complex conjugation $\Sigma$ of $\widetilde M$. In the strips, we
  set $\Sigma_{j,\tau(j)}\colon B_{j,\tau(j)}\to B_{-j,\tau(-j)}$
  defined by $\Sigma_{j,\tau(j)}(Z) = \overline Z$, which is
  well-defined since $\eta_{j,\tau(j)} =
  -\overline{\eta_{-j,\tau(-j)}}$.  These maps extend to the
  separatrices on the boundaries of the strips, and because
  $\overline{S_{-j}} = S_j$ (a consequence of $\tau(-j) =
  -\tau(j)$), this defines a global map $\Sigma$ on $M$.  Moreover,
  it maps a half infinite cylinder of $Z_{[j]}$ on a half infinite
  cylinder of $Z_{[-j]}$, so that we can extend $\Sigma$ over the
  $Z_{[j]}$'s. Lastly, it extends on $\{E_j\}_j$ by
  $\Sigma(E_j) = E_{-j}$ and maps a neighbourhood of $E$ on itself.
  Hence, $\Sigma$ is well-defined on $\widetilde M$. 

  Now, let $\sigma = h\circ \Sigma\circ h\inv$. It is an antiholomorphic
  involution of $\hat \Cc$ that maps $\infty$ on itself and $s_0$
  on itself. Then $\sigma(z) = \overline z + ia$, for $a\in\R$. The
  translation $z\mapsto {ia\over 2}$ kills $a$, so we may as well
  suppose that $a=0$. Then $\sigma = id$ only on $\R$.
  \begin{enumerate}
    \item On a strip such that $B_{-j,\tau(-j)} = \overline{B_{j,\tau(j)}}$,
      $\Sigma = id$ on the real axis. This means that this orbit of
      $1\frac\del{\del Z}$ is mapped by $h$ on an interval of $\R$, and
      the $\alpha$- and $\omega$-limits are real singular points. There are
      $m-1+\ell$ such strips.
    \item On the other strips, there are no invariant orbits under
      $\Sigma$, except possibly separatrices. 
    \begin{itemize}
      \item When $m=0$, then $S_0 = S_k = S_{-k}$ is the only invariant orbit
        under $\Sigma$, on which $\Sigma = id$, so it is mapped on $\R$ by $h$. 
      \item When $m>1$, then $S_0$ and $S_k$ are the only two invariant
        orbits under $\Sigma$. They are both mapped on an half-line of
        $\R$ with the landing points $z_{[0]},z_{[k]}\in\R$
        respectively. (If $m=1$, then $z_{[0]} = z_{[k]}$.)
    \end{itemize}
  \end{enumerate}
  In all cases, $\R$ is invariant under $P\frac\del{\del z}$, hence
  $P$ has real coefficients. With a real translation, it is now
  possible to center $P$.  \qed

%

%

\def\D{\mathbb D} 

\end{document}